\documentclass[12pt]{amsart}
\usepackage{amsmath,amssymb,amscd,amsfonts,amsmath}
\usepackage[all]{xy}

\renewcommand{\H}{{\mathbf H}}                   
\newcommand{\C}{{\mathbf C}}                   
\newcommand{\A}{{\mathbf A}^1}                   
\renewcommand{\P}{\mathbf{P}^1}                
\newcommand{\E}{{\mathcal E}}                  
\newcommand{\N}{{\mathcal N}}                  
\renewcommand{\O}{{\mathcal O}}                 
\newcommand{\Om}{{\mathbf \Omega}}                 
\newcommand{\Mod}{{\mathcal M}}               
\renewcommand{\d}{\mbox{d}}                      

\newcommand{\isom}{\cong}
\renewcommand{\to}{\longrightarrow}
\newcommand{\dbar}{\bar{\partial}}
\newcommand{\End}{{\mathcal E}nd}         
\newcommand{\Hom}{{\mathcal H}om}         
\renewcommand{\Om}{{\mathbf \Omega}}        
\renewcommand{\L}{{\mathcal L}}           

\DeclareMathOperator{\coker}{coker}       
\DeclareMathOperator{\tr}{tr}             
\DeclareMathOperator{\diag}{diag} 
\DeclareMathOperator{\res}{res} 

\DeclareMathOperator{\Sl}{Sl}
\DeclareMathOperator{\ad}{ad}
\DeclareMathOperator{\para}{par}
\DeclareMathOperator{\iso}{iso}

\newtheorem{prop}{Proposition}[section]
\newtheorem{cond}[prop]{Condition}
\newtheorem{rk}[prop]{Remark}

\newtheorem{lem}[prop]{Lemma}

\newtheorem{cor}[prop]{Corollary}
\newtheorem{thm}[prop]{Theorem}

\title[Deformations of Fuchsian equations]{Deformation theory of Fuchsian equations and logarithmic connections}
\author[Szil\'ard Szab\'o]{Szil\'ard Szab\'o }
\thanks{Department of Geometry, Technical University,
  Egry J. u. 1/H,
  Budapest 1111, Hungary,
  \texttt{szabosz@math.bme.hu}}
\date{\today}

\begin{document}

\begin{abstract}
Motivated by a remark and a question of Nicholas Katz, we characterize the tangent space of the 
space of Fuchsian equations with given generic exponents inside the corresponding 
moduli space of logarithmic connections: we construct a weight $1$ Hodge structure on the 
tangent space of the moduli of logarithmic connections such that deformations of 
Fuchsian equations correspond to the $(1,0)$-part. 
\end{abstract}

\maketitle

\section{Statement of results}
Let $p_1,\ldots ,p_n\in \P$ be $n \geq 2$ fixed points in the affine part of the complex projective line, 
and let $p_{0}$ be the point at infinity. 
Define $P$ to be the simple effective divisor $p_0 + \cdots + p_n$ in $\P$, and let $P^o=p_1 + \cdots + p_n$. 
Consider the function 
\begin{equation}\label{psi}
    \psi(z) = (z-p_1)\cdots (z-p_n)
\end{equation}
as an identification between $\O_{\P}$ and $\O_{\P}(-P^o)$ on the affine $\A=\P\setminus \{p_0 \}$. 

For $w=w(z)$ a holomorphic function of the complex variable $z$, let $w^{(k)}$ stand for its $k$-th order 
(anti-holomorphic) differential with respect to $z$ and let 
\begin{equation}\label{equation}
   w^{(m)}-\frac{G_1(z)}{\psi}w^{(m-1)}-\cdots -\frac{G_m(z)}{\psi^m}w=0,
\end{equation}
where the $G_k$ are polynomials in $z$, be a Fuchsian differential equation. 
Recall that this means that all the solutions $w$ grow at most polynomially 
with $(z-p_j)^{-1}$ near $p_j$ for any $1 \leq j \leq n$ (respectively, solutions grow at most 
polynomially with $z$ near infinity), in any given finite sector centered at the punctures. 
The left-hand side of (\ref{equation}) is a linear differential operator of the order $m$ of $w$, 
that we shall denote by $\L$. We shall identify the equation (\ref{equation}) with the operator $\L$. 
A classical result of I.L.Fuchs states that a necessary and sufficient
condition on the degrees of the $G_k$'s for $\L$ to be of Fuchsian type is
that the degree of $G_k$ has to be at most $k(n-1)$. 

Let us introduce the expressions 
\begin{align}
    w_1 & = w \notag \\
    w_2 & = \psi \frac{\d w}{\d z} \notag \\
        & \vdots \label{extension} \\
    w_m & = \psi^{m-1} \frac{\d^{m-1} w}{\d z^{m-1}}\notag . 
\end{align}
We may think of $w$ as a local section of any line bundle $L$ over $\P$ fixed in advance, 
for instance the structure sheaf $\O$. 
Then, on the affine open $\A$ the function $w_k$ is meromorphic with zeroes of order at least $(k-1)$ in $P$. 
In other words, on the affine part $\A$ a vector $(w_1,\ldots,w_m)$ is a section of the 
holomorphic vector bundle 
$$
    \tilde{E}= \O_{\P} \oplus \O_{\P}(-P^o) \oplus \ldots \oplus \O_{\P}((1-m)P^o). 
$$ 
We equip $\tilde{E}$ with the algebraic integrable connection with logarithmic poles at $P$ 
\begin{equation}\label{conn}
    D_{\L}=\d^{1,0}-\frac{A(z)}{\psi(z)} \d z, 
\end{equation}
where $A(z)$ is the {\em modified companion matrix} 
\begin{equation}\label{a}
    A  =\begin{pmatrix}
             0 & 0& 0 & 0 &\ldots \ldots & G_m \\
          1 & \psi ' & 0& 0 & \ldots \ldots &G_{m-1} \\
          0 & 1 & 2 \psi ' & 0 & \ldots \ldots &G_{m-2} \\
          \vdots  & \vdots & \vdots &   &  & \vdots  \\
          0 & 0  & 0 & \ldots   &  (m-2)\psi '& G_2 \\
          0 & 0 & 0 & \ldots & 1 & G_1 + (m-1)\psi ' 
      \end{pmatrix}
\end{equation}
of equation (\ref{equation}). Here we have denoted $\psi '=\d \psi /\d z$. 
One readily checks that a meromorphic function $w$ on some open set $U\subset \A$ with poles at most in $P$ 
locally solves (\ref{equation}) if and only if the vector $(w_1=w,w_2,\ldots ,w_m)$ is a parallel section of 
$D_{\L}$ on $\tilde{E}$ over $U$ for some (hence, only one) vector $(w_2,\ldots ,w_m)$. 

\begin{rk}
Instead of the formulae (\ref{extension}), we could have simply chosen
$w_k=w^{(k)}$, and the form of the same connection $D_{\L}$ in this
trivialisation would then be a usual companion matrix. The reason for 
our choice for the extension (\ref{extension}) is that it gives rise to 
a \emph{logarithmic} lattice. In fact, the two points of view are equivalent, so 
that a connection in modified companion form is also induced by an equation. 
\end{rk}

By assumption, the integrable connection $D_{\L}$ is regular at infinity as well. 
Therefore, according to a theorem of 
N. Katz (Thm. II.1.12 \cite{del}), there exists a lattice for the meromorphic bundle 
$$
  \N = \O_{\P}(*P) \oplus \ldots \oplus \O_{\P}(*P)
$$ 
at infinity with respect to which $D_{\L}$ is a logarithmic connection, 
i.e. its local form in any holomorphic trivialisation contains $1$-forms with at most first-order poles. 
Such a lattice at infinity can be obtained similarly to the case of the other singularities. 
For this purpose, let $\zeta=z^{-1}$ be a local coordinate at infinity. Recall that the 
first component $w=w_1$ of (\ref{extension}) is supposed to be a section of $\O_{\P}$. Then, a logarithmic 
lattice $(\tilde{w}_1,\ldots,\tilde{w}_m)$ at infinity can be obtained by the formulae 
\begin{align}
    \tilde{w}_1 & = w \notag \\
    \tilde{w}_2 & = \zeta \frac{\d w}{\d \zeta} \notag \\
        & \vdots \label{extinf} \\
    \tilde{w}_m & = \zeta^{m-1} \frac{\d^{m-1} w}{\d \zeta^{m-1}}, \notag 
\end{align}
and the form of the connection is then again a modified companion matrix with entries in the 
last column equal to the coefficients of the equation multiplied by
appropriate powers of $\zeta$ which make the connection logarithmic near infinity. 
Now one can check by induction that as a differential operator 
$$
    \zeta^{j} \frac{\d^{j}}{\d \zeta^{j}} = (-z)^{j} \frac{\d^{j}}{\d z^{j}} + 
       \sum_{\kappa=0}^{j-1}a_{\kappa j} (-z)^{\kappa} \frac{\d^{\kappa}}{\d z^{\kappa}}
$$
for some integers $a_{\kappa j}$, $0<\kappa<j<m-1$. Set then $a_{j j}=1$ and $a_{\kappa j}=0$ for $\kappa>j$. 
We see moreover that for large $z$ one has 
$$
    \psi(z) \approx z^n, 
$$
hence the trivialisations (\ref{extension}) and (\ref{extinf}) are linked on $\C^*$ by the 
matrix 
$$
    B \diag(1,-z^{n-1},\ldots,(-1)^{m-1}z^{(m-1)(n-1)}) A,
$$
with $A=(a_{\kappa j})_{\kappa,j=1}^m$ and $B$ a holomorphic matrix on $\P\setminus \{0\}$, 
invertible at $\infty$. 
It follows that we defined a logarithmic extension of the connection $D_{\L}$ on the holomorphic bundle 
\begin{equation}\label{holbdl}
    E_{\L} = \O_{\P} \oplus \O_{\P}(\infty - P^o) \oplus \ldots \oplus \O_{\P}((m-1)(\infty - P^o))
\end{equation}
over $\P$. 
Hence, (\ref{extension}) induces a bijective correspondence between local solutions of (\ref{equation}) 
and local parallel sections of the logarithmic connection $D_{\L}$ on the holomorphic bundle $E_{\L}$. 
It is not difficult to see that had we started out with taking $w$ to be a section of a non-trivial 
line bundle $L$, we would have obtained a logarithmic connection on $E_{\L}\otimes L$. 
The following fundamental result is part of folklore; however, we give a proof
in the Appendix for lack of appropriate reference. 
\begin{prop}\label{prop:gaugeequiv}
If the logarithmic connections induced by the Fuchsian equations $\L_1,\L_2$ are gauge-equivalent, 
then $\L_1=\L_2$. 
\end{prop}

Consider now the residue $\res(p,D_{\L})$ of $D_{\L}$ on $E_{\L}$ at each of the singular points $p \in P$; 
it is a well-defined endomorphism of the fiber $(E_{\L})_{p_j}$ of $E_{\L}$ at $p$, defined by applying 
$D_{\L}$ to the vector field $(z-p_j)\frac{\partial}{\partial z}$ and restricting to $(E_{\L})_{p_j}$ for 
$0<j\leq n$ (respectively, $-z^{-1}\frac{\partial}{\partial z}$ and restricting to $(E_{\L})_{p_0}$). 
Denote by $\{\mu^j_1,\ldots,\mu^j_m\}$ the set of generalized eigenvalues of the residue at $p_j$; 
this is also the set of exponents of the differential equation $\L$ at $p_j$ \cite{Sz}. 

\begin{rk}
Notice in particular that $\deg(E_{\L})=(1-n)m(m-1)/2$, in accordance with the classical Fuchs' relation, 
which states that the sum $\sum_{j,k}\mu^j_k$ of the eigenvalues of the residues of $D_{\L}$ in all 
singularities (including infinity) is equal to $(n-1)m(m-1)/2$, and with the residue theorem 
which states that the sum of the eigenvalues of the residues must be minus the degree of $E_{\L}$. 
\end{rk}

Throughout the paper, we will assume the genericity conditions : 
\begin{cond}\label{cond}
\begin{enumerate}
\item The eigenvalues $\mu^j_1,\ldots,\mu^j_m$ of the residue of the integrable connection 
$D_{\L}$ at each singularity $p_j$ do not differ by integers. 
(In particular, they are distinct.) We call this the \emph{non-resonance} condition. \label{nonres}
\item For any $k\in\{1,\ldots,m\}$ and any indices $\alpha^j_l$ for 
$j\in\{0,\ldots ,n\}$ and $l\in\{1,\ldots,k\}$ the sum $\sum_{j=0}^n\sum_{l=1}^k\mu^j_{\alpha^j_l}$ 
is not an integer. \label{nofinsum}
\end{enumerate}
\end{cond}

Condition \ref{nonres} implies that all the residues $\res(p_j,D_{\L})$ are regular semi-simple 
endomorphisms of the corresponding fibre $(E_{\L})_{p_j}$. 
Condition \ref{nofinsum} says that there can be no proper $D$-invariant subbundles of $E$, because 
by the residue theorem the negative of the degree of such a subbundle would be a sum as in the condition. 
In particular, it follows that any logarithmic connection with these eigenvalues of its residue is stable. 

Let us choose an arbitrary ordering 
of the eigenvalues $\mu^j_1,\ldots,\mu^j_m$, and denote the corresponding eigenvectors of 
$\res(p_j,D_{\L})$ by $v^j_1,\ldots,v^j_m\in(E_{\L})_{p_j}$. 
For $0\leq k\leq m $ let $E^j_k$ be the subspace of $(E_{\L})_{p_j}$ spanned by the vectors 
$\{v^j_{k+1},\ldots,v^j_m\}$. 
These subspaces define a full flag in $(E_{\L})_{p_j}$: 
$$
	(E_{\L})_{p_j} = E^j_0	\supset E^j_1 \supset \cdots \supset E^j_{m-1} \supset E^j_m = \{ 0\}. 
$$
In different terms, we get a quasi-parabolic strucutre, which is clearly compatible with 
$\res(p_j,D_{\L})$ in the sense that the residue maps each term of the filtration into itself. 

Furthermore, for any $j\in\{0,\ldots,n\}$ let us choose real numbers $\beta^j_1,\ldots,\beta^j_m$ 
such that $\beta^j_k$ and $\beta^j_l$ do not differ by an integer for any $j,k,l$ 
(in particular, they are distinct). 
We will call such a choice $\beta^j_1,\ldots,\beta^j_m$ \emph{generic parabolic weights}. 


We are interested in the following two numbers: 
\begin{enumerate}
\item the dimension $e$ of the space $\E$ of deformations of the polynomials in (\ref{equation}) 
such that all residues of the associated integrable connection $D_{\L}$ remain in the same conjugacy class 
\label{def1}
\item the dimension $c$ of the moduli space $\Mod$ of S-equivalence classes of 
$\beta$-stable integrable connections $(E,D)$ logarithmic in $P$ over a vector bundle $E$ of degree 
$d=(1-n)m(m-1)/2$, with fixed conjugacy classes of all its residues. \label{def2} 
\end{enumerate}
\begin{rk}\label{rk:moduli}
There are various sources for the definition of the moduli space $\Mod$. 
In any case, it is a quasi-projective scheme which is smooth for generic choice of 
eigenvalues $\mu$. 
An analytic construction is carried out (in greater generality than what we need 
here, namely in the case with irregular singularities) in Sections 6-8 of \cite{biqboa}. 
In the $\Sl_2$-case an algebraic construction of $\Mod$ in the relative setting 
is given in \cite{iis1}, which is generalized in \cite{ina} to the higher-rank case. 
On the other hand, the space $\E$ is well-known to be an affine space, see e.g. \cite{ince}. 
\end{rk}
It is immediate that $e\leq c$, for the space of Fuchsian deformations \ref{def1} is contained in the 
space of integrable connections \ref{def2} having the right monodromy, and by 
Proposition \ref{prop:gaugeequiv} this inclusion map is injective. 
In short, we will call deformations leaving invariant the conjugacy classes of all the residues 
\emph{locally isomonodromic}.
In the introduction of his book \cite{ka}, N. Katz computed $c$, and noticed the equality 
\begin{equation}\label{equality}
    c=2e. 
\end{equation}
In fact, both sides of this equation turn out to be 
\begin{equation}\label{exactvalue}
   2 - 2m^2 + {m(m-1)}(n+1). 
\end{equation}
He also asked whether a geometric reason underlies this equality, more precisely, whether a weight $1$ 
Hodge structure can be found on the tangent to the moduli space of integrable connections, whose $(1,0)$-part  
would give precisely the locally isomonodromic deformations of Fuchsian equations. 

We will define such a Hodge structure in Section \ref{sec:Hodgestr}: 
\begin{thm}\label{thm:main} 
Let $D_{\L}$ be an integrable connection (\ref{conn})-(\ref{a}) induced by a Fuchsian equation (\ref{equation}) 
satisfying Condition \ref{cond}. Then there exists a natural weight $1$ Hodge structure on the tangent 
at $D_{\L}$ to the moduli space $\Mod$ of integrable connections
$$
     T_{D_{\L}}\Mod = H^{1,0} \oplus H^{0,1} 
$$
with the property that its part of type $(1,0)$ is the tangent of the space $\E$ of locally isomonodromic 
deformations of the Fuchsian equation: 
$$
    T_{D_{\L}}\E = H^{1,0}.
$$ 
\end{thm}
The Hodge structure comes from a hypercohomology long exact sequence, and is well-defined on 
the tangent space at all elements of the moduli space. 
Notice that a similar exact sequence already appeared several times in the literature: 
for example, in Proposition 4.1 of \cite{nit} as well as in Proposition 6.1 of \cite{iis1}. 
In addition, in the case of rank $2$ Higgs bundles over a projective curve a similar 
splitting is used in Section 7 of \cite{hit} to obtain topological information on the moduli space. 

Notice also that our result is completely algebraic and does not involve the choice of parabolic weights; 
it would be interesting to study the dependence of these results on the weight chambers. 
In the particular case of rank $2$ bundles with $4$ parabolic points this is carried out in \cite{lss}. 


In Subsection \ref{sec:Applications}, we draw the following consequence of our Theorem: 
\begin{cor}\label{thm:Hitchinmap}
$\E$ is an algebraic subvariety which is Lagrangian with respect to the 
natural holomorphic symplectic structure of the de Rham moduli space. 
\end{cor}
We are aware that J. Aidan \cite{aidan} has results in the same direction. 

There are quite a few possible directions to generalise our results. 
First, one could study the same questions on moduli spaces of
integrable connections on higher-genus curves with marked points --- 
we plan to treat this question in a forthcoming paper \cite{ssz}. 
Similar structures for an arbitrary reductive structure group 
on higher-genus curves (but without logarithmic points) called opers were 
defined by A. Beilinson and V. Drinfeld in \cite{bei-drin}. 
This could possibly lead to a generalisation for other structure groups too. 
Another possible generalisation would be to allow for arbitrary residues. 

Notice finally that a Dolbeault analog of our embedding of the space of 
Fuchsian equations to the moduli space appeared in the article of N. Hitchin 
\cite{hit-teich} under the name Teichm\"uller component. 

\paragraph{\bf{Acknowledgments}} During the preparation of this text the author 
benefited of discussions with D. Bertrand, O. Biquard, P. Boalch, T. Hausel, 
N. Hitchin, M-H. Saito, C. Simpson, A. Stipsicz and A. Szenes; 
we would like to thank for the insightful comments and motivating discussions. 
The paper was written with the support of Hungarian Science Foundation OTKA grant NK 81203 
and \emph{Lend\"ulet project}. 

\section{The Hodge structure} \label{sec:Hodgestr}

\subsection{Construction}\label{subsec:construction} 

First, let us recall some basic facts about deformations of logarithmic connections on 
curves. Most of this material is proved in Section 12 of \cite{biq} for projective 
varieties with a smooth polar divisor. 

Let $(E,D)$ be an arbitrary element of $\Mod$. 
Let us denote by $\End_{\para}(E)$ the sheaf of parabolic endomorphisms: by definition, 
a section $\varphi$ near $p_j\in P$ of the endomorphism-bundle of $E$ is parabolic if 
$\varphi(p_j)$ respects the parabolic structure, i.e. maps each $V^j_k$ into itself. 
Furthermore, let $\End_{\iso}(E)$ denote the sheaf of locally isomonodromic endomorphisms, 
which are by definition those parabolic endomorphisms $\varphi$ whose value $\varphi(p_j)$ at 
$p_j$ lies in the adjoint orbit of the residue of the connection in the Lie algebra $\mathfrak{gl}(m)$; 
in different terms, $\varphi(p_j)$ maps each $V^j_k$ into $V^j_{k+1}$. 
By Condition \ref{cond}, the residue at $p_j$ is regular diagonal in the trivialisation 
$v^j_1,\ldots,v^j_m$; the parabolic endomorphisms are the ones whose value at $p_j$ is a lower-triangular 
matrix in this basis, and the locally isomonodromic endomorphisms are the ones whose value at 
$p_j$ is a strictly lower-triangular matrix in this basis. 
Denote by $\H^i$ the $i$-th hypercohomology of a sheaf complex. 
The infinitesimal deformations of the integrable connection $(E,D)$ (without changing the eigenvalues of the 
residues) are then described by the first hypercohomology $\H^1(\mathcal{C})$ of the two-term complex 
\begin{equation}\label{complex}
    \End_{\para}(E) \xrightarrow{\ad_D} \Om^1(P) \otimes \End_{\iso}(E), \tag{$\mathcal{C}$}
\end{equation}
where $\ad_D$ denotes the action of $D$ on endomorphisms. 
Write $\nabla=D+\dbar^E$ for the differential geometric flat connection associated to the 
couple $(E,D)$. 
We denote by $\d_{\nabla}$ the exterior derivative induced by $\nabla$ on $E$-valued 
differential forms, and extend it to every bundle functorially associated to $E$. 
If $(E,D)$ is stable, then there exists a unique harmonic metric $h$ on $E$: this
is a Hermitian metric satisfying a second-order non-linear partial
differential equation, that we do not spell out here; for details see
e.g. \cite{biq}. 
Let us denote by $^*$ the operation of taking adjoint of an endomorphism of
$E$ (or formal adjoint of a differential operator on $E$). 
Set $\Delta=\d_{\nabla}\d_{\nabla}^*+\d_{\nabla}^*\d_{\nabla}$ for the associated Laplace operator 
on differential forms, and call sections annihilated by $\Delta$ harmonic forms. 
Let us endow the base $\P\setminus P$ with a metric which looks like the
Poincar\'e metric near the punctures. 
The elements of $\H^1(\mathcal{C})$ are then represented by
endomorphism-valued $L^2$ harmonic $1$-forms. 
Furthermore, by Proposition 6.1 of \cite{iis1}, cup product on hypercohomology induces a natural 
complex symplectic structure on $\Mod$. 

The hypercohomology long exact sequence for (\ref{complex}) reads 
\begin{align}\label{exseq}
   0 \to \H^0(\mathcal{C}) & \to H^0(\End_{\para}(E)) \xrightarrow{H^0(D)} H^0(\Om^1(P) \otimes \End_{iso}(E)) \to \notag \\
     \to \H^1(\mathcal{C}) & \to H^1(\End(E)_{\para}) \xrightarrow{H^1(D)} H^1(\Om^1(P) \otimes \End_{iso}(E)) \to \notag \\
     & \to \H^2(\mathcal{C}) \to 0. 
\end{align}
The maps $H^i(D)$ are induced by $D$ on the corresponding cohomology spaces. 
Setting 
\begin{align}
    C & = \coker(H^0(D)) \\
    K & = \ker(H^1(D)), 
\end{align}
there follows a short exact sequence for the space of infinitesimal deformations: 
\begin{equation}\label{shortexsq}
   0 \to C \to \H^1(\mathcal{C}) \to K \to 0. 
\end{equation}
Consider $\mathcal{P}_d(\{\beta^j_k\})$, the moduli space of $\beta$-stable parabolic bundles of 
degree $d=(1-n)m(m-1)/2$, defined in \cite{ms}. The moduli space $\Mod$ admits a Zariski open subset 
$\Mod^0$ where the underlying parabolic bundle $E$ is stable. 
Notice that in principle for some choices of parabolic weights $\{\beta^j_k\}$ it might happen 
that there are no stable underlying bundles, i.e. $\Mod^0$ may be empty. 
In the particular case of rank $2$ bundles with $4$ parabolic points, this question is analyzed 
in Section 4 of \cite{lss}. Namely, it is proved there that there exists a non-empty set of parabolic 
weights (given by some explicit inequalities) such that the generic quasi-parabolic structure with 
these weights is stable. It would be interesting to show a similar statement in the general case too; 
in the following proposition we tacitly assume this is the case. 
\begin{prop}
The tangent of the forgetful map 
\begin{equation}\label{mp}
	\Mod^0\to\mathcal{P}_d(\{\beta^j_k\})
\end{equation}
mapping $(E,D)\in\Mod^0$ to $E$ is the map 
\begin{equation}\label{tmp}
	\H^1(\mathcal{C}) \to K
\end{equation}
of (\ref{shortexsq}). Equivalently, the term $C$ in (\ref{shortexsq}) corresponds to infinitesimal modifications 
of the logarithmic connection on a fixed parabolic bundle $E$. 
\end{prop}
\begin{proof}


Let $(\alpha,\beta)\in\H^1(\mathcal{C})$ be a class in the tangent space at $(E,D)$ of $\Mod^0$ 
with $\alpha\in H^0(\Om^1(P)\otimes\End_{\iso}(E))$ and $\beta\in H^1(\End_{\para}(E))$ and 
consider the associated infinitesimal deformation $(E_{\varepsilon},D_{\varepsilon})$. 
Then $E_{\varepsilon}$ is the holomorphic bundle with $\bar{\partial}$-operator equal to 
$\bar{\partial}^E+\varepsilon\beta$ and $D_{\varepsilon}=D+\varepsilon\alpha$. 
The cocycle-condition says that $D_{\varepsilon}$ is a logarithmic connection on $E_{\varepsilon}$. 
Notice that (\ref{tmp}) maps $(\alpha,\beta)$ to $\beta$. 
Therefore, to show that it is the tangent of (\ref{mp}) which forgets the connection $D$ there only  
remains to prove that for $\beta$-stable $E$ the tangent space $T_E\mathcal{P}_d(\{\beta^j_k\})$ 
is equal to $K$. This is essentially the content of Theorem 5.2 of \cite{ms}. 
\end{proof}

We now come back to algebraic considerations; in what follows we do not use the stability condition on $E$ any more. 
Let us next show the duality statement which will be of fundamental importance in the rest of the paper. 

\begin{lem}\label{lem:duality}
For all $(E,D)\in\Mod$ Serre duality induces a $\C$-vector space isomorphism $K^{\vee}\cong C$. 
\end{lem}
Here and in all what follows, we use the standard bilinear form 
$$
   \langle A,B\rangle = \frac{1}{m}\tr (A^tB),
$$
where $A^t$ stands for the transpose of $A$, 
to identify the Lie algebra $\mathfrak{gl}(m)$ with its dual. 
It also induces orthogonal projection operators to all vector subspaces of $\mathfrak{gl}(m)$. 
\begin{proof}

Let us compute the dual
\begin{equation}\label{dualcomplex}
    C^{-1} \xrightarrow{\ad_{D^t}} C^0 \tag{$\hat{\mathcal{C}}$}
\end{equation}
of (\ref{complex}), where the two non-zero terms are in degree $-1$ and $0$, 
and $D^t$ stands for the transpose of $D$. 
Observe that since duality changes the sign of the weights it inverts their order, 
so $D^t$ will be a parabolic map in (\ref{dualcomplex}). 
Let us determine the sheaves $C^i$. 
The dual of the sheaf $\End(E)$ is clearly $\End(E)$ itself, so away from the 
singular points both $C^{-1}$ and $C^0$ clearly coincide with $\End(E)$. 
Consider a singular point $p\in D$ and fix an arbitrary trivialisation of $E$ 
in a neighborhood of $p$ in which $\res_p(D)$ is diagonal. 
As we have seen, with respect to such a trivialisation $\End_{\para}(E)$ is the 
sheaf of endomorphisms of $E$ vanishing above the diagonal.  
Since the dual of the vanishing condition is having a simple pole, it follows that 
$C^0$ consists of local meromorphic endomorphism-valued $1$-forms near $p$ whose 
coefficients strictly above the diagonal have a pole of order at most $1$, and all 
other coefficients are regular. 
Similarly, since local sections of $\Om^1(P)\otimes\End_{\iso}(E)$ near $p$ are local 
meromorphic $1$-forms with strictly lower-triangular residue, it follows that 
$C^{-1}$ is composed of local holomorphic endomorphisms whose value at $p$ is 
upper triangular, in other words whose evaluation on the strictly lower-triangular 
coefficients vanishes. Now, one notices that $C^{-1}$ and $C^0$ are the transposes 
of $\End_{\para}(E)$ and $\Om^1(P) \otimes \End_{\iso}(E)$. This, together with the 
fact that the map in $\hat{\mathcal{C}}$ is the transpose of $D$, shows that 
transposition gives an isomorphism of complexes between $\hat{\mathcal{C}}[-1]$
and $\mathcal{C}$, so Serre duality implies the statement. 
\end{proof}

Hence, Lemma \ref{lem:duality} together with (\ref{exseq}) exhibits $\H^1(\mathcal{C})$ 
as the extension of two vector spaces of the same dimension: 
$$
    0 \to C \to \H^1(\mathcal{C}) \to K \to 0. 
$$ 
Let us show that the two factors of this extension are complex conjugate to each other. 
We first construct an anti-linear map from $C$ to $K$. 
An element of $C$ is a class of global sections of $\Om^1 (P) \otimes \End_{iso}(E)$, 
modulo the image of the map $H^0(D)$ induced by the connection on global sections of $\End(E)$. 
Let $\alpha' \d z$ represent such a class. 
Recall that we denote by $h$ the harmonic metric for the integrable connection and by $^*$ the operation of taking 
adjoint with respect to $h$. 
Then, one has $\coker(H^0(D))\isom \ker(H^0(D)^*)$, therefore there exists a unique global section 
$\alpha\d z \in \ker(H^0(D)^*)$ such that $\alpha\d z=\alpha'\d z + D f$ for some 
$f \in H^0(\End(E))$. 
Then, since $\alpha \d z$ is of type $(1,0)$, it follows that 
\begin{align*}
   \Delta(\alpha\d z) & = \d_{\nabla}^* (\dbar^E \alpha\d z)+ \d_{\nabla} (D^*\alpha\d z) \\
       & = \nabla (D^*\alpha\d z) \\
       & = D (D^*\alpha\d z) \\
       & = 0. 
\end{align*}
The second equality in this sequence holds because $\alpha\d z$ is by assumption a 
global holomorphic form, the third one follows from the map 
$$
   H^0(D)^* : H^0(\Om^1 (P) \otimes \End_{iso}(E)) \to H^0(\End(E)), 
$$
and the last one is a consequence of the assumption $\alpha\d z \in \ker(H^0(D)^*)$. 
Hence, $\alpha\d z$ is an endomorphism-valued global harmonic $(1,0)$-form representing the class 
of $\alpha'\d z$. 
Take the adjoint of the endomorphism with respect to $h$ and the complex conjugate $1$-form: 
$\alpha^*\d \bar{z}$. 
According to the $L^2$ Dolbeault resolution (Lemma 9.1 and Theorem 5.1, \cite{biq}) 
it defines an element $[\alpha^*\d \bar{z}]$ of $H^1(\P,\End(E))$. 
Since $\Delta$ is a real operator, we deduce that 
\begin{align*}
   0=\Delta(\alpha^*\d \bar{z}) & =\d_{\nabla}^* (D \alpha^*\d \bar{z}) + 
         \nabla ((\dbar^E)^* \alpha^*\d \bar{z}) \\
      & = \d_{\nabla}^* (D \alpha^*\d \bar{z}), 
\end{align*}
where the second equality follows from $\dbar^E \alpha\d z=0$. 
By irreducibility of $D$, the $2$-form $D \alpha^*\d \bar{z}$ must be trivial, 
i.e. the cohomology class $[\alpha^*\d \bar{z}] \in H^1(\P,\End(E))$ is in $\ker(H^1(D))=K$. 
Call this cohomology class $\varsigma([\alpha\d z])$. Clearly, $\varsigma$ is then an anti-linear map 
from $C$ to $K$. In the next paragraph we will see that it is bijective. 

The map $\varsigma$ can be extended to define a conjugation on the whole tangent 
space $\H^1(\mathcal{C})$. 
Indeed, by Theorem 12.6 of \cite{biq} the hypercohomology $\H^1(\mathcal{C})$ is isomorphic to 
the first $L^2$-cohomology of $\nabla$ acting on endomorphisms of $E$ with
respect to $h$, or said differently, to the $L^2$-kernel of the Laplacian
$\Delta$ acting on endomorphisms. 
Let $\alpha\d z +\beta\d \bar{z}$ be the endomorphism-valued $L^2$ harmonic $1$-form 
representing a given class $a\in\H^1(\mathcal{C})$. 
Since the Laplacian is a real operator, the $1$-form $\beta^*\d z +\alpha^*\d \bar{z}$ 
is also harmonic, hence it represents a class in $\H^1(\mathcal{C})$. We declare this class 
to be $\varsigma(a)$. This map is clearly involutive, in particular bijective. 

The conjugation $\varsigma$ coupled with the symplectic structure induces a 
skew-Hermitian form $g$ on $\H^1(\mathcal{C})$ as follows. 
Let $\alpha_j\d z +\beta_j\d \bar{z}$ be the harmonic representatives of the classes $a_j\in \H^1(\mathcal{C})$ 
for $j \in \{1,2\}$. Then $g$ is defined by the formula 
\begin{align}\label{Hermitian}
    g(a_1,a_2)& = ia_1 \cup \varsigma(a_2) \\   
  & = i\int_{\P}\tr(\alpha_1\d z\wedge \alpha_2^*\d \bar{z}+\beta_1\d \bar{z}\wedge\beta_2^*\d z). \notag
\end{align}
Clearly, $C$ and $K$ are orthogonal complements of each other with respect to
this metric. 
Moreover, the form $ig$ is positive definite on $C$ and negative definite on $K$. 
In particular, we deduce the orthogonal decomposition of vector spaces 
\begin{align}\label{decomp}
     \H^1(\mathcal{C}) & = C \oplus K \\
              & = H^{1,0} \oplus H^{0,1}. \notag
\end{align}
which is just the decomposition of harmonic $1$-forms according to type. 
The second line of this decomposition together with $\varsigma$ 
define a weight $1$ Hodge structure on the tangent of $\Mod$, and the Hermitian 
metric (\ref{Hermitian}) induces a polarisation on it. 
In the next subsection we show that it admits the desired property. 

\subsection{Characterisation of deformations of the Fuchsian equation}\label{subsec:char}
In this subsection, we show that the Hodge structure defined in the previous subsection satisfies 
the property claimed in Theorem \ref{thm:main}. In all this part, we consider a point 
$(E_{\L},D_{\L})\in \E$, induced by the Fuchsian equation $\L$. 

First, notice that it is sufficient to show that for any Fuchsian equation 
$\L$ with the right exponents, the parabolic bundle $E_{\L}$ is always the same. 
Indeed, as we have seen in Subsection \ref{subsec:construction}, 
this then implies that an infinitesimal deformation of Fuchsian equations always 
lies inside the $(1,0)$-part $C$ of the Hodge structure (at least in the case where 
this parabolic bundle is stable). In different terms, we then have 
 $T_{D_{\L}}\E\subset C$. On the other hand, by Lemma \ref{lem:duality} 
and (\ref{equality}) the dimensions of the vector subspaces $T_{D_{\L}}\E$ and 
$C$ of  $T_{D_{\L}}\Mod$ are both of dimension $e$, so they must agree. 

Therefore, suppose we are given any polynomials $G_1,\ldots G_m$ in the variable $z$ such that 
$\deg(G_k)\leq k(n-1)$ and that the exponents of the corresponding equation $\L$ at any  
puncture $p_j$ are the numbers $\mu^j_1,\ldots,\mu ^j_m$ we fixed in advance. 
It is clear that the lattices (\ref{extension}) and (\ref{extinf}) are independent of the choice 
of such polynomials, just as the gluing matrices between them. On the other hand, the residue 
at any $p_j$ of $D_{\L}$ is equal to a modified companion matrix $A(p_j)$ where $A$ is defined in 
(\ref{a}). In particular, for any $1\leq k\leq m$ its $\mu^j_k$-eigenspace only depends on the 
residue, and so is also independent of the polynomials $G_k$ themselves. 
However, the eigenvalues at $p_j$ are assumed to be different, and the parabolic structure of 
$E$ is induced by the eigenvectors of the residue; hence, the parabolic structure is independent 
of the polynomials $G_k$ in the same trivialisation which was already seen to be independent of the 
choice of $G_k$. To sum up, the underlying parabolic bundle $E_{\L}$ is independent of the 
choice of the polynomials $G_k$ satisfying the property that the exponents of $\L$ at all $p^j$ are  
$\mu^j_1,\ldots,\mu ^j_m$. 

\subsection{Lagrangian property}\label{sec:Applications}
Let us denote by $I$ and $J$ the de Rham and Dolbeault complex structures on $\Mod$ respectively, 
and set $K=IJ$ \cite{hit}. Further, we denote by $\omega_I,\omega_J,\omega_K$ the associated 
symplectic structures and set $\Omega_I=-\omega_J+i\omega_K$. We then have: 
\begin{cor}\label{cor:etale}
$\E$ is a complex algebraic Lagrangian submanifold of $\Mod$ with respect to the holomorphic 
symplectic form $\Omega_I$. 
\end{cor}
\begin{rk}
The Lagrangian property is proved independently in the recent work \cite{aidan} by Jonathan Aidan. 
His proof relies on matrix commutator computations, hence is different from ours. 
\end{rk}
\begin{proof}
Let us consider the canonical family 
\begin{equation}\label{canfam}
  (E_{\L},D_{\L})\longrightarrow\E
\end{equation} 
of logarithmic connections on $\P$ parameterised by $\E$. 
By construction (\ref{extension},\ref{conn},\ref{a},\ref{extinf}) the bundle 
$E_{\L}$ is independent of $\L$ and the connection matrix of $D_{\L}$ with 
respect to the independent trivialisation depends algebraically on the 
coefficients of the polynomials $G_1,\ldots,G_m$. 
It follows that the family (\ref{canfam}) is an algebraic family. 
By the universal property of $\Mod$, the inclusion map 
$\E\hookrightarrow\Mod$ is then algebraic. 

Recall that the tangent space of $\Mod$ is identified with endomorphism-valued $L^2$ harmonic $1$-forms. 
For such $1$-forms $\phi_1,\phi_2$ the holomorphic symplectic form of the de Rham moduli 
space can be written as 
$$
    \Omega_I(\phi_1,\phi_2) = \int_{\C} \tr (\phi_1 \wedge \phi_2). 
$$
Clearly this quantity vanishes if both $\phi_1$ and $\phi_2$ are of type $(1,0)$. 
\end{proof}


\section{Appendix: The proof of Proposition \ref{prop:gaugeequiv}}
\begin{proof}
Suppose there exists a gauge transformation $g \in \Hom(E_{\L_1},E_{\L_2})$ mapping $D_{\L_1}$ into $D_{\L_2}$. 
In the decompositions (\ref{holbdl}) of $E_{\L_1}$ and $E_{\L_2}$, $g$ can be written as a matrix whose 
entry in the $k$-th row and $l$-th column is a global holomorphic section of the sheaf $\O((k-l)(n-1))$. 
It follows that the matrix of $g$ is lower triangular, and that the entries on the diagonal are 
global sections of the trivial holomorphic line bundle over $\P$, hence constants. 
For $j=1,2$ let us write on the affine part $\C$ of $\P$ away from infinity the expressions 
$$
    D_{\L_j}=\d^{1,0}-\frac{A_j(z)}{\psi(z)} \d z.  
$$
It is then a well-known fact that the action of $g$ on $D_{\L_1}$ is 
$$
   g \cdot (D_{\L_1}) = \d^{1,0} - \frac{g^{-1}A_1(z)g}{\psi(z)} \d z 
     - g^{-1}d^{1,0}g.
$$
It follows from the above observations that $g^{-1}d^{1,0}g$ is strictly lower triagular.  
In particular, the entries on and above the diagonal in the matrices $A_1$ and $A_2$ must agree. 

Let us first consider the case $m=2$. Here, one has 
$$
   g = \begin{pmatrix}
       g_{11} & 0 \\
       g_{21} & g_{22}
       \end{pmatrix}, 
$$
where $g_{21}$ is a global section of $\O(n-1)$ and $g_{11},g_{22}$ are constants with $g_{11}g_{22}\neq 0$, 
and the inverse of this matrix is 
$$
   g^{-1} = \frac{1}{g_{11}g_{22}} 
       \begin{pmatrix}
       g_{22} & 0 \\
       -g_{21} & g_{11}
       \end{pmatrix}. 
$$
Furthermore, the matrices of the equations are 
$$
   A_j= \begin{pmatrix}
       0 & 1 \\
       G^j_2 & G^j_1 + \psi'
     \end{pmatrix}, 
$$
where $G^j_2, G^j_1$ are the coefficients of $\L_j$. 
Straightforward matrix multiplication yields 
$$
   g^{-1}A_1(z)g = \frac{1}{g_{11}g_{22}} 
    \begin{pmatrix}
      g_{21}g_{22} & g_{22}^2 \\
      * & * 
       \end{pmatrix}. 
$$
By the above, the terms in the first row must be equal to $0$ and $1$ respectively. 
We infer that $g_{21}=0$ and $g_{22}=g_{11}$, hence $g$ is a multiple of the identity. 

We now come to the general case. As the computations are more involved but 
of the same spirit, we only sketch the proof. The matrix $g$ is equal to 
$$
   g = \begin{pmatrix}
       g_{11} & 0 & \ldots & 0 \\
       g_{21} & g_{22}  & \ldots & 0 \\ 
       \vdots & & \ddots & \vdots \\
       g_{m1} & g_{m2} & \ldots & g_{mm} 
       \end{pmatrix}, 
$$
and its inverse is of the form 
$$
    g^{-1}=\frac{1}{g_{11}g_{22}\cdots g_{mm}}
      \begin{pmatrix}
      g_{22}\cdots g_{mm} & 0 & \ldots & 0 \\
      * & g_{11}g_{33}\cdots g_{mm}& \ldots & 0 \\
      \vdots & & \ddots & \vdots \\
      * & * & \ldots & g_{11}\cdots g_{m-1,m-1} 
      \end{pmatrix}. 
$$
One has 
\begin{align*}
   g^{-1}A_1(z)g & = \frac{1}{g_{11}g_{22}\cdots g_{mm}} \cdot \\
   & \begin{pmatrix}
      g_{22}\cdots g_{mm} & 0 & \ldots & 0 \\
      * & g_{11}g_{33}\cdots g_{mm}& \ldots & \vdots \\
      \vdots & & \ddots & 0 \\
      * & * & \ldots & g_{11}\cdots g_{m-1,m-1} 
   \end{pmatrix} \\ 
   & \begin{pmatrix}
       g_{21} & g_{22} & 0 & \ldots & 0 \\
       g_{21}\psi'+g_{31} & g_{22}\psi'+g_{32} & g_{33} & \ldots & \vdots \\ 
       \vdots & & \ddots & \ddots & 0 \\
       * & & & g_{m-1,m-1}\psi'+g_{m,m-1} & g_{mm}\\
       * & \ldots & & * & * 
   \end{pmatrix}.
\end{align*}
As the entries on and above the diagonal in this product have to be equal to those of the modified 
companion matrix $A_2(z)$, we deduce as before that $g_{11}=g_{22}=\cdots =g_{mm}$ and 
$g_{21}=g_{32}=\cdots =g_{m,m-1}=0$. 
It follows that right below the diagonal all the entries of the matrix $g^{-1}d^{1,0}g$ vanish. 
Considering now the first sub-diagonal in the product above, we deduce that $g_{31}=g_{42}=\cdots=g_{m,m-2}=0$.
It follows that on the second sub-diagonal of the matrix $g^{-1}d^{1,0}g$ all the entries vanish. 
Continuing this argument, we eventually obtain that all the $g_{kl}$ for $k>l$ must vanish. 
This concludes the proof. 
\end{proof}

\bibliographystyle{alpha}
\bibliography{deformation-theory}

\begin{thebibliography}{LSS10}

\bibitem[Aid]{aidan}
Jonathan Aidan.
\newblock {\em Propri\'et\'es symplectiques de l'espace des \'equations
  differentielles dans l'espace des syst\`emes logarithmiques}.
\newblock PhD thesis, Universit\'e Pierre et Marie Curie, Paris.
\newblock In preparation.

\bibitem[BB04]{biqboa}
Olivier Biquard and Philip Boalch.
\newblock Wild non-abelian {H}odge theory on curves.
\newblock {\em Compositio Mathematica}, 140(1):179--204, 2004.

\bibitem[BD05]{bei-drin}
Alexander Beilinson and Vladimir Drinfeld.
\newblock Opers.
\newblock arXiv:math/0501398, 2005.

\bibitem[Biq97]{biq}
Olivier Biquard.
\newblock Fibr\'es de {H}iggs et connexions int\'egrables: le cas logarithmique
  (diviseur lisse).
\newblock {\em Annales scientifiques de l'\'Ecole Normale Sup\'erieure},
  30(4):41--96, 1997.

\bibitem[Del70]{del}
Pierre Deligne.
\newblock {\em \'Equations diff\'erentielles \`a points singuliers
  r\'eguliers}, volume 163 of {\em Lecture Notes in Mathematics}.
\newblock Springer-Verlag, 1970.

\bibitem[Hit87]{hit}
Nigel~J. Hitchin.
\newblock The self-duality equations on a {R}iemann surface.
\newblock {\em Proceedings of the London Mathematical Society}, 55(3):59--126,
  1987.

\bibitem[Hit92]{hit-teich}
Nigel~J. Hitchin.
\newblock Lie groups and {T}eichm\"uller space.
\newblock {\em Topology}, 31(3):449--473, 1992.

\bibitem[IIS06]{iis1}
Michi-Aki Inaba, Katsunori Iwasaki, and Masa-Hiko Saito.
\newblock Moduli of stable parabolic connections, {R}iemann-{H}ilbert
  correspondence and geometry of {P}ainlev\'e equation of type {VI}, part {I}.
\newblock {\em Publications of the Research Institute for Mathematical
  Sciences}, 42(4), 2006.

\bibitem[Ina06]{ina}
Michi-Aki Inaba.
\newblock Moduli of parabolic connections on a curve and {R}iemann-{H}ilbert
  correspondence.
\newblock arXiv:math/0602004, 2006.

\bibitem[Inc26]{ince}
E.~L. Ince.
\newblock {\em Ordinary Differential Equations}.
\newblock Dover Publications, New York, 1926.

\bibitem[Kat96]{ka}
Nicholas~M. Katz.
\newblock {\em Rigid Local Systems}.
\newblock Number 139 in Annals of Mathematics Studies. Princeton University
  Press, 1996.

\bibitem[LSS10]{lss}
Frank Loray, Masa-Hiko Saito, and Carlos Simpson.
\newblock Foliations on the moduli space of rank two connections on the
  projective line minus four points.
\newblock arXiv:1012.3612, 2010.

\bibitem[MS80]{ms}
V.~B. Mehta and C.~S. Seshadri.
\newblock Moduli of vector bundles on curves with parabolic structures.
\newblock {\em Mathematische Annalen}, 248:205--239, 1980.

\bibitem[Nit93]{nit}
Nitin Nitsure.
\newblock Moduli of semistable logarithmic connections.
\newblock {\em Journal of the American Mathematical Society}, 6:597--609, 1993.

\bibitem[SS]{ssz}
M-H. Saito and Sz. Szab\'o.
\newblock Apparent singularities and canonical coordinates for moduli of
  connections.
\newblock In preparation.

\bibitem[Sza08]{Sz}
Szil\'ard Szab\'o.
\newblock The extension of a {F}uchsian equation onto the projective line.
\newblock {\em Acta Scientiarum Mathematicarum (Szeged)}, 74:557--564, 2008.

\end{thebibliography}

\end{document}